\documentclass[12pt,twoside]{amsart}
\usepackage[latin1]{inputenc}
\usepackage{amsmath, amsthm, amscd, amsfonts, amssymb, graphicx}
\usepackage[bookmarksnumbered, plainpages]{hyperref}

\newtheorem{thm}{Theorem}[section]
\newtheorem{cor}[thm]{Corollary}
\newtheorem{lem}[thm]{Lemma}

\newtheorem{defn}[thm]{Definition}
\newtheorem{rem}[thm]{\bf{Remark}}

\numberwithin{equation}{section}

\begin{document}

\title{\bf A note on modular forms and generalized anomaly cancellation formulas \uppercase\expandafter{\romannumeral 2}}
\author{Siyao Liu \hskip 0.4 true cm  Yong Wang$^{*}$}

\thanks{{\scriptsize
\hskip -0.4 true cm \textit{2010 Mathematics Subject Classification:}
58C20; 57R20; 53C80.
\newline \textit{Key words and phrases:} Modular invariance; Cancellation formulas.
\newline \textit{$^{*}$Corresponding author}}}

\maketitle

\begin{abstract}
 \indent In \cite{LW}, Liu and Wang generalized the Han-Liu-Zhang cancellation formulas to the $(a, b)$ type cancellation formulas. In this note, we prove some another $(a, b)$ type cancellation formulas for even-dimensional Riemannian manifolds. And by transgression, we obtain some characteristic forms with modularity properties on odd-dimensional manifolds.
\end{abstract}

\vskip 0.2 true cm

%------------------------------------------------------------------------------------%

\pagestyle{myheadings}
\markboth{\rightline {\scriptsize Liu}}
         {\leftline{\scriptsize Generalized cancellation formulas and modular forms}}

\bigskip
\bigskip

%------------------------------------------------------------------------------------%
%------------------------------------------------------------------------------------%

\section{ Introduction }

In \cite{AW}, Alvarez-Gaum\'{e} and Witten discovered the ``miraculous cancellation" formula for gravitational anomaly  as follow:
\begin{align}
\Big\{\widehat{L}(TM, \nabla^{TM})\Big\}^{(12)}=\Big\{\widehat{A}(TM, \nabla^{TM})ch(T_{\mathbb{C}}M, \nabla^{T_{\mathbb{C}}M})-32\widehat{A}(TM, \nabla^{TM})\Big\}^{(12)},\nonumber
\end{align}
where $T_{\mathbb{C}}M$ denotes the complexification of $TM$ and $\nabla^{T_{\mathbb{C}}M}$ is canonically induced from $\nabla^{TM},$ the Levi-Civita connection associated to the Riemannian structure of $M.$
This formula reveals a beautiful relation between the top components of the Hirzebruch $\widehat{L}$-form and $\widehat{A}$-form of a $12$-dimensional smooth Riemannian manifold $M.$
Liu established higher-dimensional ``miraculous cancellation" formulas for $(8k+4)$-dimensional Riemannian manifolds by developing modular invariance properties of characteristic forms \cite{L}.
Han and Zhang established a general cancellation formula that involves a complex line bundle for $(8k+4)$-dimensional smooth Riemannian manifold in \cite{HZ1,HZ2}.
For higher-dimensional smooth Riemannian manifolds the authors obtained some cancellation formulas in \cite{HH}.
In \cite{W1}, Wang proved more general cancellation formulas for $(8k+2)$ and $(8k+6)$-dimensional smooth Riemannian manifolds.
And in \cite{LW}, the authors generalized the Han-Liu-Zhang cancellation formulas to the $(a, b)$ type cancellation formulas.

Moreover, motivated by the Chern-Simons theory, Chen and Han computed the transgressed forms of some modularly invariant characteristic forms, which are related to the elliptic genera and studied the modularity properties of these secondary characteristic forms and relations among them \cite{CH}.
In \cite{W2}, Wang computed the transgressed forms of some modularly invariant characteristic forms related to the ``twisted" elliptic genera.
The motivation of this paper is to prove more $(a, b)$ type cancellation formulas.

A brief description of the organization of this paper is as follows.
In Section 2, we give some definitions and basic notions that we will use in this paper.
In the next section, we prove some generalized cancellation formulas for $4d$-dimensional Riemannian manifolds.
In Section 4, for even-dimensional Riemannian manifolds, we give some generalized cancellation formulas involving a complex line bundle.
Finally, in Section 5, we obtain some characteristic forms with modularity properties for odd-dimensional manifolds.

%------------------------------------------------------------------------------------%

\vskip 1 true cm

\section{ Characteristic forms and modular forms }

Firstly, we give some definitions and basic notions on characteristic forms and modular forms that will be used throughout the paper.
For the details, see \cite{A,H,Z}.

\vskip 0.3 true cm
2.1. \noindent{ Characteristic forms }
\vskip 0.3 true cm

Let $M$ be a Riemannian manifold, $\nabla^{TM}$ be the associated Levi-Civita connection on $TM$ and $R^{TM}=(\nabla^{TM})^{2}$ be the curvature of $\nabla^{TM}.$
According to the detailed descriptions in \cite{Z}, let $\widehat{A}(TM, \nabla^{TM})$ and $\widehat{L}(TM, \nabla^{TM})$ be the Hirzebruch characteristic forms defined respectively by
\begin{align}
\widehat{A}(TM, \nabla^{TM})&=\det\nolimits^{\frac{1}{2}}\left(\frac{\frac{\sqrt{-1}}{4\pi}R^{TM}}{\sinh(\frac{\sqrt{-1}}{4\pi}R^{TM})}\right).
\end{align}

Let $E, F$ be two Hermitian vector bundles over $M$ carrying Hermitian connection $\nabla^{E}, \nabla^{F}$ respectively.
Let $R^{E}=(\nabla^{E})^{2}$ (resp. $R^{F}=(\nabla^{F})^{2}$) be the curvature of $\nabla^{E}$ (resp. $\nabla^{F}$).
If we set the formal difference $G=E-F,$ then $G$ carries an induced Hermitian connection $\nabla^{G}$ in an obvious sense.
We define the associated Chern character form as
\begin{align}
\text{ch}(G, \nabla^{G})=\text{tr}\left[\exp\left(\frac{\sqrt{-1}}{2\pi}R^{E}\right)\right]-\text{tr}\left[\exp\left(\frac{\sqrt{-1}}{2\pi}R^{F}\right)\right].
\end{align}

For any complex number $t,$ let
\begin{align}
\wedge_{t}(E)&=\mathbb{C}|_{M}+tE+t^{2}\wedge^{2}(E)+\cdot\cdot\cdot,\\
S_{t}(E)&=\mathbb{C}|_{M}+tE+t^{2}S^{2}(E)+\cdot\cdot\cdot,
\end{align}
denote respectively the total exterior and symmetric powers of $E,$ which live in $K(M)[[t]].$
The following relations between these operations hold,
\begin{align}
S_{t}(E)=\frac{1}{\wedge_{-t}(E)},~~~~ \wedge_{t}(E-F)=\frac{\wedge_{t}(E)}{\wedge_{t}(F)}.
\end{align}
Moreover, if $\{ \omega_{i} \}, \{ \omega'_{j} \}$ are formal Chern roots for Hermitian vector bundles $E, F$ respectively, then
\begin{align}
\text{ch}(\wedge_{t}(E))=\prod_{i}(1+e^{\omega_{i}}t).
\end{align}
Then we have the following formulas for Chern character forms,
\begin{align}
\text{ch}(S_{t}(E))=\frac{1}{\prod\limits_{i}(1-e^{\omega_{i}}t)},~~~ \text{ch}(\wedge_{t}(E-F))=\frac{\prod\limits_{i}(1+e^{\omega_{i}}t)}{\prod\limits_{j}(1+e^{\omega'_{j}}t)}.
\end{align}

If $W$ is a real Euclidean vector bundle over $M$ carrying a Euclidean connection $\nabla^{W},$ then its complexification $W_{\mathbb{C}}=W\otimes \mathbb{C}$ is a complex vector bundle over $M$ carrying a canonical induced Hermitian metric from that of $W,$ as well as a Hermitian connection $\nabla^{W_{\mathbb{C}}}$ induced from $\nabla^{W}.$
If $E$ is a vector bundle (complex or real) over $M,$ set $\widetilde{E}=E-\dim E$ in $K(M)$ or $KO(M).$

\vskip 0.3 true cm
2.2. \noindent{ Some properties about the Jacobi theta functions and modular forms }
\vskip 0.3 true cm

Refer to \cite{C}, we recall four Jacobi theta functions are defined as follows:
\begin{align}
&\theta(v, \tau)=2q^{\frac{1}{8}}\sin(\pi v)\prod^{\infty}_{j=1}[(1-q^{j})(1-e^{2\pi\sqrt{-1}v}q^{j})(1-e^{-2\pi\sqrt{-1}v}q^{j})];\\
&\theta_{1}(v, \tau)=2q^{\frac{1}{8}}\cos(\pi v)\prod^{\infty}_{j=1}[(1-q^{j})(1+e^{2\pi\sqrt{-1}v}q^{j})(1+e^{-2\pi\sqrt{-1}v}q^{j})];\\
&\theta_{2}(v, \tau)=\prod^{\infty}_{j=1}[(1-q^{j})(1-e^{2\pi\sqrt{-1}v}q^{j-\frac{1}{2}})(1-e^{-2\pi\sqrt{-1}v}q^{j-\frac{1}{2}})];\\
&\theta_{3}(v, \tau)=\prod^{\infty}_{j=1}[(1-q^{j})(1+e^{2\pi\sqrt{-1}v}q^{j-\frac{1}{2}})(1+e^{-2\pi\sqrt{-1}v}q^{j-\frac{1}{2}})],
\end{align}
where $q=e^{2\pi\sqrt{-1}\tau}$ with $\tau\in\mathbb{H},$ the upper half complex plane.

Let
\begin{align}
\theta'(0, \tau)=\frac{\partial \theta(v, \tau)}{\partial v}\Big|_{v=0}.
\end{align}
Then the following Jacobi identity holds,
\begin{align}
\theta'(0, \tau)=\pi\theta_{1}(0, \tau)\theta_{2}(0, \tau)\theta_{3}(0, \tau).
\end{align}

In what follows,
\begin{align}
SL_{2}(\mathbb{Z})=\Big\{ \Big( \begin{array}{cc}
a & b\\
c & d
\end{array}
\Big)\Big|a, b, c, d\in\mathbb{Z}, ad-bc=1
\Big \}
\end{align}
stands for the modular group.
Write $S=\Big(\begin{array}{cc}
0 & -1\\
1 & 0
\end{array}\Big),
T=\Big(\begin{array}{cc}
1 & 1\\
0 & 1
\end{array}\Big)$
be the two generators of $SL_{2}(\mathbb{Z}).$
We choose the sign such that for $\tau=\sqrt{-1}, \big(\frac{\tau}{\sqrt{-1}}\big)^{\frac{1}{2}}=+1.$
They act on $\mathbb{H}$ by $S\tau=-\frac{1}{\tau}, T\tau=\tau+1.$
One has the following transformation laws of theta functions under the actions of $S$ and $T$ (see \cite{L,CH}):
\begin{align}
&\theta(v, \tau+1)=e^{\frac{\pi\sqrt{-1}}{4}}\theta(v, \tau),~~~~ \theta(v, -\frac{1}{\tau})=\frac{1}{\sqrt{-1}}\Big(\frac{\tau}{\sqrt{-1}}\Big)^{\frac{1}{2}}e^{\pi\sqrt{-1}\tau v^{2}}\theta(\tau v, \tau);\\
&\theta_{1}(v, \tau+1)=e^{\frac{\pi\sqrt{-1}}{4}}\theta_{1}(v, \tau),~~~~ \theta_{1}(v, -\frac{1}{\tau})=\Big(\frac{\tau}{\sqrt{-1}}\Big)^{\frac{1}{2}}e^{\pi\sqrt{-1}\tau v^{2}}\theta_{2}(\tau v, \tau);\\
&\theta_{2}(v, \tau+1)=\theta_{3}(v, \tau),~~~~ \theta_{2}(v, -\frac{1}{\tau})=\Big(\frac{\tau}{\sqrt{-1}}\Big)^{\frac{1}{2}}e^{\pi\sqrt{-1}\tau v^{2}}\theta_{1}(\tau v, \tau);\\
&\theta_{3}(v, \tau+1)=\theta_{2}(v, \tau),~~~~ \theta_{3}(v, -\frac{1}{\tau})=\Big(\frac{\tau}{\sqrt{-1}}\Big)^{\frac{1}{2}}e^{\pi\sqrt{-1}\tau v^{2}}\theta_{3}(\tau v, \tau).
\end{align}
Differentiating the above transformation formulas, we get that
\begin{align}
&\theta'(v, \tau+1)=e^{\frac{\pi\sqrt{-1}}{4}}\theta'(v, \tau),\\
&\theta'(v, -\frac{1}{\tau})=\frac{1}{\sqrt{-1}}\Big(\frac{\tau}{\sqrt{-1}}\Big)^{\frac{1}{2}}e^{\pi\sqrt{-1}\tau v^{2}}(2\pi\sqrt{-1}\tau v\theta(\tau v, \tau)+\tau\theta'(\tau v, \tau));\nonumber\\
&\theta_{1}'(v, \tau+1)=e^{\frac{\pi\sqrt{-1}}{4}}\theta_{1}'(v, \tau),\\
&\theta_{1}'(v, -\frac{1}{\tau})=\Big(\frac{\tau}{\sqrt{-1}}\Big)^{\frac{1}{2}}e^{\pi\sqrt{-1}\tau v^{2}}(2\pi\sqrt{-1}\tau v\theta_{2}(\tau v, \tau)+\tau\theta'_{2}(\tau v, \tau));\nonumber\\
&\theta'_{2}(v, \tau+1)=\theta'_{3}(v, \tau),\\
&\theta'_{2}(v, -\frac{1}{\tau})=\Big(\frac{\tau}{\sqrt{-1}}\Big)^{\frac{1}{2}}e^{\pi\sqrt{-1}\tau v^{2}}(2\pi\sqrt{-1}\tau v\theta_{1}(\tau v, \tau)+\tau\theta'_{1}(\tau v, \tau));\nonumber\\
&\theta'_{3}(v, \tau+1)=\theta'_{2}(v, \tau),\\
&\theta'_{3}(v, -\frac{1}{\tau})=\Big(\frac{\tau}{\sqrt{-1}}\Big)^{\frac{1}{2}}e^{\pi\sqrt{-1}\tau v^{2}}(2\pi\sqrt{-1}\tau v\theta_{3}(\tau v, \tau)+\tau\theta'_{3}(\tau v, \tau)).\nonumber
\end{align}
Therefore
\begin{align}
\theta'(0, -\frac{1}{\tau})=\frac{1}{\sqrt{-1}}\Big(\frac{\tau}{\sqrt{-1}}\Big)^{\frac{1}{2}}\tau\theta'(0, \tau).
\end{align}

\begin{defn}
A modular form over $\Gamma,$ a modular subgroup of $SL_{2}(\mathbb{Z}),$ is a holomorphic function $f(\tau)$ on $\mathbb{H}$ such that
\begin{align}
f(g\tau):=f\Big(\frac{a\tau+b}{c\tau+d}\Big)=\chi(g)(c\tau+d)^{k}f(\tau),~\forall g=\Big(\begin{array}{cc}
a & b\\
c & d
\end{array}\Big)\in\Gamma,
\end{align}
where $\chi:\Gamma\rightarrow\mathbb{C}^{*}$ is a character of $\Gamma,$ $k$ is called the weight of $f.$
\end{defn}

Let
\begin{align}
&\Gamma_{0}(2)=\Big\{ \Big( \begin{array}{cc}
a & b\\
c & d
\end{array}
\Big)\in SL_{2}(\mathbb{Z})\Big|c\equiv0(\text{mod}~2)
\Big \},\\
&\Gamma^{0}(2)=\Big\{ \Big( \begin{array}{cc}
a & b\\
c & d
\end{array}
\Big)\in SL_{2}(\mathbb{Z})\Big|b\equiv0(\text{mod}~2)
\Big \},\\
&\Gamma_{\theta}=\Big\{ \Big( \begin{array}{cc}
a & b\\
c & d
\end{array}
\Big)\in SL_{2}(\mathbb{Z})\Big|\Big( \begin{array}{cc}
a & b\\
c & d
\end{array}
\Big)\equiv\Big( \begin{array}{cc}
1 & 0\\
0 & 1
\end{array}
\Big) \text{or} \Big( \begin{array}{cc}
0 & 1\\
1 & 0
\end{array}
\Big)(\text{mod}~2)
\Big \}
\end{align}
be the three modular subgroups of $SL_{2}(\mathbb{Z}).$
It is known that the generators of $\Gamma_{0}(2)$ are $T, ST^{2}ST,$ the generators of $\Gamma^{0}(2)$ are $STS, T^{2}STS$ and the generators of $\Gamma_{\theta}$ are $ S, T^{2}.$

Let $E_{2}(\tau)$ be the Eisenstein series which is a quasimodular form over $SL_{2}(\mathbb{Z}),$ satisfying
\begin{align}
E_{2}\Big(\frac{a\tau+b}{c\tau+d}\Big)=(c\tau+d)^{2}E_{2}(\tau)-\frac{6\sqrt{-1}c(c\tau+d)}{\pi}.
\end{align}
In particular, we have
\begin{align}
&E_{2}(\tau+1)=E_{2}(\tau),\\
&E_{2}\Big(-\frac{1}{\tau}\Big)=\tau^{2}E_{2}(\tau)-\frac{6\sqrt{-1}\tau}{\pi}.
\end{align}

If $\Gamma$ is a modular subgroup, let $\mathcal{M}_{\mathbb{R}}(\Gamma)$ denote the ring of modular forms over $\Gamma$ with real Fourier coefficients.
Writing $\theta_{j}=\theta_{j}(0,\tau), 1\leq j\leq3,$ we introduce six explicit modular forms,
\begin{align}
&\delta_{1}(\tau)=\frac{1}{8}(\theta^{4}_{2}+\theta^{4}_{3}),~~~~\varepsilon_{1}(\tau)=\frac{1}{16}\theta^{4}_{2}\theta^{4}_{3};\\
&\delta_{2}(\tau)=-\frac{1}{8}(\theta^{4}_{1}+\theta^{4}_{3}),~~~~\varepsilon_{2}(\tau)=\frac{1}{16}\theta^{4}_{1}\theta^{4}_{3};\\
&\delta_{3}(\tau)=\frac{1}{8}(\theta^{4}_{1}-\theta^{4}_{2}),~~~~\varepsilon_{3}(\tau)=-\frac{1}{16}\theta^{4}_{1}\theta^{4}_{2}.
\end{align}
They have the following Fourier expansions in $q^{\frac{1}{2}}$:
\begin{align}
&\delta_{1}(\tau)=\frac{1}{4}+6q+\cdot\cdot\cdot,~~~~\varepsilon_{1}(\tau)=\frac{1}{16}-q+\cdot\cdot\cdot;\\
&\delta_{2}(\tau)=-\frac{1}{8}-3q^{\frac{1}{2}}+\cdot\cdot\cdot,~~~~\varepsilon_{2}(\tau)=q^{\frac{1}{2}}+\cdot\cdot\cdot;\\
&\delta_{3}(\tau)=-\frac{1}{8}+3q^{\frac{1}{2}}+\cdot\cdot\cdot,~~~~\varepsilon_{3}(\tau)=-q^{\frac{1}{2}}+\cdot\cdot\cdot,
\end{align}
where the $``\cdot\cdot\cdot"$ terms are the higher degree terms, all of which have integral coefficients.
They also satisfy the transformation laws,
\begin{align}
&\delta_{2}\Big(-\frac{1}{\tau}\Big)=\tau^{2}\delta_{1}(\tau),~~~~\varepsilon_{2}\Big(-\frac{1}{\tau}\Big)=\tau^{4}\varepsilon_{1}(\tau);\\
&\delta_{2}(\tau+1)=\delta_{3}(\tau),~~~~\varepsilon_{2}(\tau+1)=\varepsilon_{3}(\tau).
\end{align}

\begin{lem}(\cite{L})
$\delta_{1}(\tau)$ (resp. $\varepsilon_{1}(\tau)$) is a modular form of weight 2 (resp. 4) over $\Gamma_{0}(2),$ $\delta_{2}(\tau)$ (resp. $\varepsilon_{2}(\tau)$) is a modular form of weight 2 (resp. 4) over $\Gamma^{0}(2),$ while $\delta_{3}(\tau)$ (resp. $\varepsilon_{3}(\tau)$) is a modular form of weight 2 (resp. 4) over $\Gamma_{\theta},$ and moreover $\mathcal{M}_{\mathbb{R}}(\Gamma^{0}(2))=\mathbb{R}[\delta_{2}(\tau), \varepsilon_{2}(\tau)].$
\end{lem}

Let $E$ be a vector bundle and $f$ be a power series with constant term $1.$
Let $\nabla_{t}^{E}$ be deformed connection given by $\nabla_{t}^{E}=(1-t)\nabla_{0}^{E}+t\nabla_{1}^{E}$ and $R^{E}_{t}, t\in [0,1],$ denote the curvature of $\nabla_{t}^{E}.$ $f'(t)$ is the power series obtained from the derivative of $f(x)$ with respect to $x.$
$\omega$ is a closed form.
Then we have

\begin{lem}(\cite{CH})
\begin{align}
\det\nolimits^{\frac{1}{2}}(f(R^{E}_{1}))\omega-\det\nolimits^{\frac{1}{2}}(f(R^{E}_{0}))\omega=d\int_0^1\frac{1}{2}\det\nolimits^{\frac{1}{2}}(f(R^{E}_{t}))\omega tr\Big[\frac{d\nabla_{t}^{E}}{dt}\frac{f'(R^{E}_{t})}{f(R^{E}_{t})}\Big]dt.
\end{align}
\end{lem}

\section{ The generalized cancellation formulas for even-dimensional Riemannian manifolds }

Set $M$ be a $4d$-dimensional Riemannian manifold with the associated Levi-Civita connection $\nabla^{TM}.$
Let $\xi$ be a complex line bundle and $\xi_{\mathbb{R}}$ be a rank two real oriented Euclidean vector bundle over $M$ carrying with a Euclidean connection $\nabla^{\xi_{\mathbb{R}}}$ which is the real bundle associated to $\xi.$
Write $a=(a_{1},\cdot\cdot\cdot, a_{\iota}), b=(b_{1},\cdot\cdot\cdot, b_{\iota}),$ where $a_{\iota}, b_{\iota}, 1\leqslant \iota\leqslant k$ are integers.

Set
\begin{align}
\Theta_{1}(T_{\mathbb{C}}M, \xi_{\mathbb{R}}, a, b)=&\bigotimes^{\infty}_{n=1}S_{q^{n}}(\widetilde{T_{\mathbb{C}}M})\otimes\bigotimes^{\infty}_{m_{1}=1}\wedge_{q^{m_{1}}}(\widetilde{(\xi^{\otimes a_{1}})_{\mathbb{R}}\otimes\mathbb{C}})\otimes\cdot\cdot\cdot\otimes\bigotimes^{\infty}_{m_{k}=1}\wedge_{q^{m_{k}}}(\widetilde{(\xi^{\otimes a_{k}})_{\mathbb{R}}\otimes\mathbb{C}})\\
&\otimes\bigotimes^{\infty}_{r_{1}=1}\wedge_{q^{r_{1}-\frac{1}{2}}}(\widetilde{(\xi^{\otimes b_{1}})_{\mathbb{R}}\otimes\mathbb{C}})\otimes\cdot\cdot\cdot\otimes\bigotimes^{\infty}_{r_{k}=1}\wedge_{q^{r_{k}-\frac{1}{2}}}(\widetilde{(\xi^{\otimes b_{k}})_{\mathbb{R}}\otimes\mathbb{C}})\nonumber\\
&\otimes\bigotimes^{\infty}_{s_{1}=1}\wedge_{-q^{s_{1}-\frac{1}{2}}}(\widetilde{(\xi^{\otimes b_{1}})_{\mathbb{R}}\otimes\mathbb{C}})\otimes\cdot\cdot\cdot\otimes\bigotimes^{\infty}_{s_{k}=1}\wedge_{-q^{s_{k}-\frac{1}{2}}}(\widetilde{(\xi^{\otimes b_{k}})_{\mathbb{R}}\otimes\mathbb{C}}),\nonumber\\
\Theta_{2}(T_{\mathbb{C}}M, \xi_{\mathbb{R}}, a, b)=&\bigotimes^{\infty}_{n=1}S_{q^{n}}(\widetilde{T_{\mathbb{C}}M})\otimes\bigotimes^{\infty}_{m_{1}=1}\wedge_{q^{m_{1}}}(\widetilde{(\xi^{\otimes b_{1}})_{\mathbb{R}}\otimes\mathbb{C}})\otimes\cdot\cdot\cdot\otimes\bigotimes^{\infty}_{m_{k}=1}\wedge_{q^{m_{k}}}(\widetilde{(\xi^{\otimes b_{k}})_{\mathbb{R}}\otimes\mathbb{C}})\\
&\otimes\bigotimes^{\infty}_{r_{1}=1}\wedge_{q^{r_{1}-\frac{1}{2}}}(\widetilde{(\xi^{\otimes b_{1}})_{\mathbb{R}}\otimes\mathbb{C}})\otimes\cdot\cdot\cdot\otimes\bigotimes^{\infty}_{r_{k}=1}\wedge_{q^{r_{k}-\frac{1}{2}}}(\widetilde{(\xi^{\otimes b_{k}})_{\mathbb{R}}\otimes\mathbb{C}})\nonumber\\
&\otimes\bigotimes^{\infty}_{s_{1}=1}\wedge_{-q^{s_{1}-\frac{1}{2}}}(\widetilde{(\xi^{\otimes a_{1}})_{\mathbb{R}}\otimes\mathbb{C}})\otimes\cdot\cdot\cdot\otimes\bigotimes^{\infty}_{s_{k}=1}\wedge_{-q^{s_{k}-\frac{1}{2}}}(\widetilde{(\xi^{\otimes a_{k}})_{\mathbb{R}}\otimes\mathbb{C}}).\nonumber
\end{align}
Obviously, $\Theta_{1}(T_{\mathbb{C}}M, \xi_{\mathbb{R}}, a, b)$ and $\Theta_{2}(T_{\mathbb{C}}M, \xi_{\mathbb{R}}, a, b)$ admit formal Fourier expansion in $q^{\frac{1}{2}}$ as
\begin{align}
\Theta_{1}(T_{\mathbb{C}}M, \xi_{\mathbb{R}}, a, b)=A_{0}(T_{\mathbb{C}}M, \xi_{\mathbb{R}}, a, b)+A_{1}(T_{\mathbb{C}}M, \xi_{\mathbb{R}}, a, b)q^{\frac{1}{2}}+\cdot\cdot\cdot,\\
\Theta_{2}(T_{\mathbb{C}}M, \xi_{\mathbb{R}}, a, b)=B_{0}(T_{\mathbb{C}}M, \xi_{\mathbb{R}}, a, b)+B_{1}(T_{\mathbb{C}}M, \xi_{\mathbb{R}}, a, b)q^{\frac{1}{2}}+\cdot\cdot\cdot,
\end{align}
where $A_{j}$ and $B_{j}$ are elements in the semi-group formally generated by Hermitian vector bundles over $M.$

Let $\{ \pm2\pi\sqrt{-1}x_{j}|1\leq j\leq 2d\},$ $\{ \pm a_{\iota}2\pi\sqrt{-1}u|1\leq \iota\leq k\}$ and $\{ \pm b_{\iota}2\pi\sqrt{-1}u|1\leq \iota\leq k\}$ be the Chern roots of $T_{\mathbb{C}}M,$ $(\xi^{\otimes a_{\iota}})_{\mathbb{R}}\otimes\mathbb{C}$ and $(\xi^{\otimes b_{\iota}})_{\mathbb{R}}\otimes\mathbb{C}$ respectively and $c=2\pi\sqrt{-1}u.$
Let $p_{1}$ denote the first Pontryagin form.
If $\omega$ is a differential form over $M,$ we denote by $\omega^{(i)}$ its degree $i$ component.
We can now state our main result of this section as follows.

\begin{thm}
The following equation holds,
\begin{align}
&\Big\{e^{\frac{1}{24}[p_{1}(TM)-\sum^{k}_{\iota=1}(a_{\iota}^{2}+2b_{\iota}^{2})p_{1}(\xi_{\mathbb{R}})]}2^{k}\cosh\Big(\frac{a_{1}}{2}c\Big)\cdot\cdot\cdot\cosh\Big(\frac{a_{k}}{2}c\Big)\widehat{A}(TM, \nabla^{TM})\Big\}^{(4d)}=\nonumber\\
&\sum_{r=0}^{[\frac{d}{2}]}2^{d-6r}\Big\{e^{\frac{1}{24}[p_{1}(TM)-\sum^{k}_{\iota=1}(a_{\iota}^{2}+2b_{\iota}^{2})p_{1}(\xi_{\mathbb{R}})]}2^{k}\cosh\Big(\frac{b_{1}}{2}c\Big)\cdot\cdot\cdot\cosh\Big(\frac{b_{k}}{2}c\Big)\widehat{A}(TM, \nabla^{TM})ch(\alpha_{r})\Big\}^{(4d)},\nonumber
\end{align}
where
\begin{align}
&\alpha_{0}=(-1)^{d}\mathbb{C},\\
&\alpha_{1}=(-1)^{d}\Big(-24d+\sum_{\iota=1}^{k}\big(\widetilde{(\xi^{\otimes b_{\iota}})_{\mathbb{R}}\otimes\mathbb{C}}-\widetilde{(\xi^{\otimes a_{\iota}})_{\mathbb{R}}\otimes\mathbb{C}}\big)\Big).
\end{align}
\end{thm}
\begin{proof}
Set
\begin{align}
Q_{1}(\tau)=&\Big\{e^{\frac{1}{24}E_{2}(\tau)[p_{1}(TM)-\sum^{k}_{\iota=1}(a_{\iota}^{2}+2b_{\iota}^{2})p_{1}(\xi_{\mathbb{R}})]}2^{k}\cosh\Big(\frac{a_{1}}{2}c\Big)\cdot\cdot\cdot\cosh\Big(\frac{a_{k}}{2}c\Big)\\
&\times\widehat{A}(TM, \nabla^{TM})ch(\Theta_{1}(T_{\mathbb{C}}M, \xi_{\mathbb{R}}, a, b))\Big\}^{(4d)},\nonumber\\
Q_{2}(\tau)=&\Big\{e^{\frac{1}{24}E_{2}(\tau)[p_{1}(TM)-\sum^{k}_{\iota=1}(a_{\iota}^{2}+2b_{\iota}^{2})p_{1}(\xi_{\mathbb{R}})]}2^{k}\cosh\Big(\frac{b_{1}}{2}c\Big)\cdot\cdot\cdot\cosh\Big(\frac{b_{k}}{2}c\Big)\\
&\times\widehat{A}(TM, \nabla^{TM})ch(\Theta_{2}(T_{\mathbb{C}}M, \xi_{\mathbb{R}}, a, b))\Big\}^{(4d)}.\nonumber
\end{align}
Via simple calculations, we can obtain
\begin{align}
Q_{1}(\tau)=&2^{k}\bigg\{e^{\frac{1}{24}E_{2}(\tau)[p_{1}(TM)-\sum^{k}_{\iota=1}(a_{\iota}^{2}+2b_{\iota}^{2})p_{1}(\xi_{\mathbb{R}})]}\Big(\prod_{j=1}^{2d}x_{j}\frac{\theta'(0, \tau)}{\theta(x_{j}, \tau)}\Big)\\
&\times\Big(\prod_{\iota=1}^{k}\frac{\theta_{1}(a_{\iota}u, \tau)}{\theta_{1}(0, \tau)}\frac{\theta_{3}(b_{\iota}u, \tau)}{\theta_{3}(0, \tau)}\frac{\theta_{2}(b_{\iota}u, \tau)}{\theta_{2}(0, \tau)}\Big)\bigg\}^{(4d)},\nonumber
\end{align}
\begin{align}
Q_{2}(\tau)=&2^{k}\bigg\{e^{\frac{1}{24}E_{2}(\tau)[p_{1}(TM)-\sum^{k}_{\iota=1}(a_{\iota}^{2}+2b_{\iota}^{2})p_{1}(\xi_{\mathbb{R}})]}\Big(\prod_{j=1}^{2d}x_{j}\frac{\theta'(0, \tau)}{\theta(x_{j}, \tau)}\Big)\\
&\times\Big(\prod_{\iota=1}^{k}\frac{\theta_{1}(b_{\iota}u, \tau)}{\theta_{1}(0, \tau)}\frac{\theta_{3}(b_{\iota}u, \tau)}{\theta_{3}(0, \tau)}\frac{\theta_{2}(a_{\iota}u, \tau)}{\theta_{2}(0, \tau)}\Big)\bigg\}^{(4d)}.\nonumber
\end{align}

By direct verifications in applying the transformation laws (2.15)-(2.18), we get
\begin{align}
Q_{1}(T\tau)&=Q_{1}(\tau+1)=Q_{1}(\tau),\\
Q_{1}(ST^{2}ST\tau)&=(T^{2}ST\tau)^{2d}Q_{2}(T^{2}ST\tau)=(T^{2}ST\tau)^{2d}Q_{3}(TST\tau)\\
&=(2\tau+1)^{2d}Q_{1}(\tau).\nonumber
\end{align}
We find $Q_{1}(\tau)$ is a modular form of weight $2d$ over $\Gamma_{0}(2),$ while $Q_{2}(\tau)$ is a modular form of weight $2d$ over $\Gamma^{0}(2).$
Moreover, the following identity holds,
\begin{align}
x_{j}\frac{\theta'(0, -\frac{1}{\tau})}{\theta(x_{j}, -\frac{1}{\tau})}=x_{j}\frac{\tau \theta'(0, \tau)}{e^{\pi\sqrt{-1}\tau x^{2}_{j}}\theta(\tau x_{j}, \tau)};\\
\frac{\theta_{1}(a_{\iota}u, -\frac{1}{\tau})}{\theta_{1}(0, -\frac{1}{\tau})}=\frac{e^{a^{2}_{\iota}\pi\sqrt{-1}\tau u^{2}}\theta_{2}(\tau a_{\iota}u, \tau)}{\theta_{2}(0, \tau)};\\
\frac{\theta_{3}(b_{\iota}u, -\frac{1}{\tau})}{\theta_{3}(0, -\frac{1}{\tau})}=\frac{e^{b^{2}_{\iota}\pi\sqrt{-1}\tau u^{2}}\theta_{3}(\tau b_{\iota}u, \tau)}{\theta_{3}(0, \tau)};\\
\frac{\theta_{2}(b_{\iota}u, -\frac{1}{\tau})}{\theta_{2}(0, -\frac{1}{\tau})}=\frac{e^{b^{2}_{\iota}\pi\sqrt{-1}\tau u^{2}}\theta_{1}(\tau b_{\iota}u, \tau)}{\theta_{1}(0, \tau)},
\end{align}
we conclude that,
\begin{align}
Q_{1}\Big(-\frac{1}{\tau}\Big)=\tau^{2d}Q_{2}(\tau).
\end{align}

Observe that at any point $x\in M,$ up to the volume form determined by the metric on $T_{x}M,$ both $\Theta_{2}(T_{\mathbb{C}}M, \xi_{\mathbb{R}}, a, b)$ and $Q_{i}(\tau), i=1,2$ can be viewed as a power series of $q^{\frac{1}{2}}$ with real Fourier coefficients.
We have
\begin{align}
Q_{2}(\tau)=h_{0}(8\delta_{2}(\tau))^{d}+h_{1}(8\delta_{2}(\tau))^{d-2}\varepsilon_{2}(\tau)+\cdot\cdot\cdot+h_{[\frac{d}{2}]}(8\delta_{2}(\tau))^{d-2[\frac{d}{2}]}\varepsilon_{2}(\tau)^{[\frac{d}{2}]}.
\end{align}
It is easily seen that
\begin{align}
\Theta_{2}(T_{\mathbb{C}}M, \xi_{\mathbb{R}}, a, b)=1+\sum_{\iota=1}^{k}\big(\widetilde{(\xi^{\otimes b_{\iota}})_{\mathbb{R}}\otimes\mathbb{C}}-\widetilde{(\xi^{\otimes a_{\iota}})_{\mathbb{R}}\otimes\mathbb{C}}\big)q^{\frac{1}{2}}+\cdot\cdot\cdot.
\end{align}
Direct computations show that
\begin{align}
Q_{2}(\tau)=&\Big\{e^{\frac{1}{24}(1-24q+\cdot\cdot\cdot)[p_{1}(TM)-\sum^{k}_{\iota=1}(a_{\iota}^{2}+2b_{\iota}^{2})p_{1}(\xi_{\mathbb{R}})]}2^{k}\cosh\Big(\frac{b_{1}}{2}c\Big)\cdot\cdot\cdot\cosh\Big(\frac{b_{k}}{2}c\Big)\\
&\times\widehat{A}(TM, \nabla^{TM})ch(1+\sum_{\iota=1}^{k}\big(\widetilde{(\xi^{\otimes b_{\iota}})_{\mathbb{R}}\otimes\mathbb{C}}-\widetilde{(\xi^{\otimes a_{\iota}})_{\mathbb{R}}\otimes\mathbb{C}}\big)q^{\frac{1}{2}}+\cdot\cdot\cdot)\Big\}^{(4d)}\nonumber\\
=&(-1)^{d}h_{0}+\big(24d(-1)^{d}h_{0}+(-1)^{d-2}h_{1}\big)q^{\frac{1}{2}}+\cdot\cdot\cdot,\nonumber
\end{align}
then we have
\begin{align}
h_{r}=\Big\{e^{\frac{1}{24}[p_{1}(TM)-\sum^{k}_{\iota=1}(a_{\iota}^{2}+2b_{\iota}^{2})p_{1}(\xi_{\mathbb{R}})]}2^{k}\cosh\Big(\frac{b_{1}}{2}c\Big)\cdot\cdot\cdot\cosh\Big(\frac{b_{k}}{2}c\Big)\widehat{A}(TM, \nabla^{TM})ch(\alpha_{r})\Big\}^{(4d)}.
\end{align}
Consequently,
\begin{align}
&\alpha_{0}=(-1)^{d}\mathbb{C},\\
&\alpha_{1}=(-1)^{d}\Big(-24d+\sum_{\iota=1}^{k}\big(\widetilde{(\xi^{\otimes b_{\iota}})_{\mathbb{R}}\otimes\mathbb{C}}-\widetilde{(\xi^{\otimes a_{\iota}})_{\mathbb{R}}\otimes\mathbb{C}}\big)\Big).
\end{align}
According to (3.17), we have
\begin{align}
Q_{1}(\tau)=&\frac{1}{\tau^{2d}}Q_{2}\Big(-\frac{1}{\tau}\Big)\\
=&\frac{1}{\tau^{2d}}\Big[h_{0}\Big(8\delta_{2}\Big(-\frac{1}{\tau}\Big)\Big)^{d}+h_{1}\Big(8\delta_{2}\Big(-\frac{1}{\tau}\Big)\Big)^{d-2}\varepsilon_{2}\Big(-\frac{1}{\tau}\Big)+\cdot\cdot\cdot\nonumber\\
&+h_{[\frac{d}{2}]}\Big(8\delta_{2}\Big(-\frac{1}{\tau}\Big)\Big)^{d-2[\frac{d}{2}]}\varepsilon_{2}\Big(-\frac{1}{\tau}\Big)^{[\frac{d}{2}]}\Big]\nonumber\\
=&h_{0}(8\delta_{1}(\tau))^{d}+h_{1}(8\delta_{1}(\tau))^{d-2}\varepsilon_{1}(\tau)+\cdot\cdot\cdot+h_{[\frac{d}{2}]}(8\delta_{1}(\tau))^{d-2[\frac{d}{2}]}\varepsilon_{1}(\tau)^{[\frac{d}{2}]}.\nonumber
\end{align}
By comparing the constant term in (3.24), we can get Theorem 3.1.
\end{proof}

\begin{cor}
When $\dim M=4,$ the following identity holds,
\begin{align}
&\Big\{2^{k}\cosh\Big(\frac{a_{1}}{2}c\Big)\cdot\cdot\cdot\cosh\Big(\frac{a_{k}}{2}c\Big)\widehat{A}(TM, \nabla^{TM})\Big\}^{(4)}+2\Big\{2^{k}\cosh\Big(\frac{b_{1}}{2}c\Big)\cdot\cdot\cdot\cosh\Big(\frac{b_{k}}{2}c\Big)\nonumber\\
&\times\widehat{A}(TM, \nabla^{TM})\Big\}^{(4)}=-\frac{2^{k}}{8}[p_{1}(TM)-\sum^{k}_{\iota=1}(a_{\iota}^{2}+2b_{\iota}^{2})p_{1}(\xi_{\mathbb{R}})].\nonumber
\end{align}
\end{cor}

\begin{cor}
If $M$ is a closed spin manifold with $\dim M=4,$ $p_{1}(TM)=\sum^{k}_{\iota=1}(a_{\iota}^{2}+2b_{\iota}^{2})p_{1}(\xi_{\mathbb{R}})$ and if $b_{1}+\cdot\cdot\cdot+b_{k}$ is an even number, then
\begin{align}
\Big\{2^{k}\cosh\Big(\frac{a_{1}}{2}c\Big)\cdot\cdot\cdot\cosh\Big(\frac{a_{k}}{2}c\Big)\widehat{A}(TM, \nabla^{TM})\Big\}[M]\nonumber
\end{align}
is a multiple of $2.$
\end{cor}
\begin{proof}
It is easily seen that
\begin{align}
2^{k}\cosh\Big(\frac{b_{1}}{2}c\Big)\cdot\cdot\cdot\cosh\Big(\frac{b_{k}}{2}c\Big)&=2^{k}\cdot\frac{e^{\frac{b_{1}}{2}c}+e^{-\frac{b_{1}}{2}c}}{2}\cdot\cdot\cdot\frac{e^{\frac{b_{k}}{2}c}+e^{-\frac{b_{k}}{2}c}}{2}\\
&=(e^{\frac{b_{1}}{2}c}+e^{-\frac{b_{1}}{2}c})\cdot\cdot\cdot (e^{\frac{b_{k}}{2}c}+e^{-\frac{b_{k}}{2}c})\nonumber\\
&=e^{-\frac{b_{1}+\cdot\cdot\cdot+b_{k}}{2}c}(e^{b_{1}c}+1)\cdot\cdot\cdot (e^{b_{k}c}+1).\nonumber
\end{align}
If $b_{1}+\cdot\cdot\cdot+b_{k}$ is an even number, we get
\begin{align}
e^{-\frac{b_{1}+\cdot\cdot\cdot+b_{k}}{2}c}&=ch((\xi^{*})^{\otimes\frac{b_{1}+\cdot\cdot\cdot+b_{k}}{2}}),\\
e^{b_{1}c}+1&=ch(\xi^{\otimes b_{1}}+\mathbb{C}),
\end{align}
then
\begin{align}
&\mathrm{Ind}\Big\{D\otimes(\xi^{*})^{\otimes\frac{b_{1}+\cdot\cdot\cdot+b_{k}}{2}}\otimes\Big(\xi^{\otimes b_{1}}+\mathbb{C}\Big)\otimes\cdot\cdot\cdot\otimes\Big(\xi^{\otimes b_{k}}+\mathbb{C}\Big)\Big\}_{+}\\
&=\Big\{2^{k}\cosh\Big(\frac{b_{1}}{2}c\Big)\cdot\cdot\cdot\cosh\Big(\frac{b_{k}}{2}c\Big)\widehat{A}(TM, \nabla^{TM})\Big\}[M]\nonumber
\end{align}
is an integer number, where $D$ is the spin Dirac operator.
So by the Corollary 3.2, we get Corollary 3.3.
\end{proof}

\begin{rem}
Notably, since $\Big\{2^{k}\cosh\Big(\frac{a_{1}}{2}c\Big)\cdot\cdot\cdot\cosh\Big(\frac{a_{k}}{2}c\Big)\widehat{A}(TM, \nabla^{TM})\Big\}[M]$ is an integer, we do not expand $\Big\{2^{k}\cosh\Big(\frac{a_{1}}{2}c\Big)\cdot\cdot\cdot\cosh\Big(\frac{a_{k}}{2}c\Big)\widehat{A}(TM, \nabla^{TM})\Big\}^{(4)}$ and preserves the factor $2^{k}$ on both sides of the equation, yielding the divisibility property in Corollary 3.3. The expression $\Big\{2^{k}\cosh\Big(\frac{a_{1}}{2}c\Big)\cdot\cdot\cdot\cosh\Big(\frac{a_{k}}{2}c\Big)\widehat{A}(TM, \nabla^{TM})\Big\}^{(4)}$ is clean.
\end{rem}

\begin{cor}
When $\dim M=8,$ the following identity holds,
\begin{align}
&\Big\{2^{k}\cosh\Big(\frac{a_{1}}{2}c\Big)\cdot\cdot\cdot\cosh\Big(\frac{a_{k}}{2}c\Big)\widehat{A}(TM, \nabla^{TM})\Big\}^{(8)}-\Big\{2^{k}\cosh\Big(\frac{b_{1}}{2}c\Big)\cdot\cdot\cdot\cosh\Big(\frac{b_{k}}{2}c\Big)\nonumber\\
&\times\widehat{A}(TM, \nabla^{TM})\Big\}^{(8)}-\frac{1}{16}\Big\{2^{k}\cosh\Big(\frac{b_{1}}{2}c\Big)\cdot\cdot\cdot\cosh\Big(\frac{b_{k}}{2}c\Big)\widehat{A}(TM, \nabla^{TM})ch\Big(\sum_{\iota=1}^{k}\big(\widetilde{(\xi^{\otimes b_{\iota}})_{\mathbb{R}}\otimes\mathbb{C}}\nonumber\\
&-\widetilde{(\xi^{\otimes a_{\iota}})_{\mathbb{R}}\otimes\mathbb{C}}\big)\Big)\Big\}^{(8)}=-\frac{1}{24}\Big\{[p_{1}(TM)-\sum^{k}_{\iota=1}(a_{\iota}^{2}+2b_{\iota}^{2})p_{1}(\xi_{\mathbb{R}})]2^{k}\cosh\Big(\frac{a_{1}}{2}c\Big)\cdot\cdot\cdot\cosh\Big(\frac{a_{k}}{2}c\Big)\nonumber\\
&\times\widehat{A}(TM, \nabla^{TM})\Big\}^{(8)}+\frac{1}{24}\Big\{[p_{1}(TM)-\sum^{k}_{\iota=1}(a_{\iota}^{2}+2b_{\iota}^{2})p_{1}(\xi_{\mathbb{R}})]2^{k}\cosh\Big(\frac{b_{1}}{2}c\Big)\cdot\cdot\cdot\cosh\Big(\frac{b_{k}}{2}c\Big)\nonumber
\end{align}
\begin{align}
&\times\widehat{A}(TM, \nabla^{TM})\Big\}^{(8)}
+\frac{1}{384}\Big\{[p_{1}(TM)-\sum^{k}_{\iota=1}(a_{\iota}^{2}+2b_{\iota}^{2})p_{1}(\xi_{\mathbb{R}})]2^{k}\cosh\Big(\frac{b_{1}}{2}c\Big)\cdot\cdot\cdot\cosh\Big(\frac{b_{k}}{2}c\Big)\nonumber\\
&\times\widehat{A}(TM, \nabla^{TM})ch\Big(\sum_{\iota=1}^{k}\big(\widetilde{(\xi^{\otimes b_{\iota}})_{\mathbb{R}}\otimes\mathbb{C}}-\widetilde{(\xi^{\otimes a_{\iota}})_{\mathbb{R}}\otimes\mathbb{C}}\big)\Big)\Big\}^{(8)}
+\frac{1}{18432}\Big\{([p_{1}(TM)-\sum^{k}_{\iota=1}(a_{\iota}^{2}\nonumber\\
&+2b_{\iota}^{2})p_{1}(\xi_{\mathbb{R}})])^{2}2^{k}ch\Big(\sum_{\iota=1}^{k}\big(\widetilde{(\xi^{\otimes b_{\iota}})_{\mathbb{R}}\otimes\mathbb{C}}-\widetilde{(\xi^{\otimes a_{\iota}})_{\mathbb{R}}\otimes\mathbb{C}}\big)\Big)\Big\}^{(8)}.\nonumber
\end{align}
\end{cor}

\section{  The generalized cancellation formulas involving a complex line bundle for even-dimensional Riemannian manifolds }
Let $\eta$ be a rank two oriented Euclidean vector bundle over $M$ carrying a Euclidean connection $\nabla^{\eta}.$
Assume that
\begin{align}
\overline{\Theta_{1}}(T_{\mathbb{C}}M, \xi_{\mathbb{R}}, \eta_{\mathbb{C}}, a, b)=&\bigotimes^{\infty}_{n=1}S_{q^{n}}(\widetilde{T_{\mathbb{C}}M})\otimes\bigotimes^{\infty}_{m=1}\wedge_{q^{m}}(\widetilde{(\xi^{\otimes a})_{\mathbb{R}}\otimes\mathbb{C}}-2\widetilde{\eta_{\mathbb{C}}})\\
&\otimes\bigotimes^{\infty}_{r=1}\wedge_{q^{r-\frac{1}{2}}}(\widetilde{(\xi^{\otimes b})_{\mathbb{R}}\otimes\mathbb{C}}+\widetilde{\eta_{\mathbb{C}}})\otimes\bigotimes^{\infty}_{s=1}\wedge_{-q^{s-\frac{1}{2}}}(\widetilde{(\xi^{\otimes b})_{\mathbb{R}}\otimes\mathbb{C}}+\widetilde{\eta_{\mathbb{C}}}),\nonumber\\
\overline{\Theta_{2}}(T_{\mathbb{C}}M, \xi_{\mathbb{R}}, \eta_{\mathbb{C}}, a, b)=&\bigotimes^{\infty}_{n=1}S_{q^{n}}(\widetilde{T_{\mathbb{C}}M})\otimes\bigotimes^{\infty}_{m=1}\wedge_{q^{m}}(\widetilde{(\xi^{\otimes b})_{\mathbb{R}}\otimes\mathbb{C}}+\widetilde{\eta_{\mathbb{C}}})\\
&\otimes\bigotimes^{\infty}_{r=1}\wedge_{q^{r-\frac{1}{2}}}(\widetilde{(\xi^{\otimes b})_{\mathbb{R}}\otimes\mathbb{C}}+\widetilde{\eta_{\mathbb{C}}})\otimes\bigotimes^{\infty}_{s=1}\wedge_{-q^{s-\frac{1}{2}}}(\widetilde{(\xi^{\otimes a})_{\mathbb{R}}\otimes\mathbb{C}}-2\widetilde{\eta_{\mathbb{C}}}).\nonumber
\end{align}

Let $\{ \pm2\pi\sqrt{-1}\overline{u}\}$  be the Chern roots of $\eta_{\mathbb{C}}$ and $\overline{c}=2\pi\sqrt{-1}\overline{u}.$
We can claim the following theorem.

\begin{thm}
The following equation holds,
\begin{align}
&\Big\{e^{\frac{1}{24}[p_{1}(TM)-(a^{2}+2b^{2})p_{1}(\xi_{\mathbb{R}})]}2\cosh\Big(\frac{a}{2}c\Big)\frac{1}{\cosh^{2}\Big(\frac{1}{2}\overline{c}\Big)}
\widehat{A}(TM, \nabla^{TM})\Big\}^{(4d)}=\nonumber\\
&\sum_{r=0}^{[\frac{d}{2}]}2^{d-6r}\Big\{e^{\frac{1}{24}[p_{1}(TM)-(a^{2}+2b^{2})p_{1}(\xi_{\mathbb{R}})]}2\cosh\Big(\frac{b}{2}c\Big)\cosh\Big(\frac{1}{2}\overline{c}\Big)\widehat{A}(TM, \nabla^{TM})ch(\overline{\alpha_{r}})\Big\}^{(4d)},\nonumber
\end{align}
where
\begin{align}
&\overline{\alpha_{0}}=(-1)^{d}\mathbb{C},\\
&\overline{\alpha_{1}}=(-1)^{d}\big(-24d+\widetilde{(\xi^{\otimes b})_{\mathbb{R}}\otimes\mathbb{C}}-\widetilde{(\xi^{\otimes a})_{\mathbb{R}}\otimes\mathbb{C}}+3\widetilde{\eta_{\mathbb{C}}}\big).
\end{align}
\end{thm}

\begin{proof}
Let
\begin{align}
\overline{Q_{1}}(\tau)=&\Big\{e^{\frac{1}{24}E_{2}(\tau)[p_{1}(TM)-(a^{2}+2b^{2})p_{1}(\xi_{\mathbb{R}})]}2\cosh\Big(\frac{a}{2}c\Big)\frac{1}{\cosh^{2}\Big(\frac{1}{2}\overline{c}\Big)}\\
&\times\widehat{A}(TM, \nabla^{TM})ch(\overline{\Theta_{1}}(T_{\mathbb{C}}M, \xi_{\mathbb{R}}, \eta_{\mathbb{C}}, a, b))\Big\}^{(4d)},\nonumber\\
\overline{Q_{2}}(\tau)=&\Big\{e^{\frac{1}{24}E_{2}(\tau)[p_{1}(TM)-(a^{2}+2b^{2})p_{1}(\xi_{\mathbb{R}})]}2\cosh\Big(\frac{b}{2}c\Big)\cosh\Big(\frac{1}{2}\overline{c}\Big)\\
&\times\widehat{A}(TM, \nabla^{TM})ch(\overline{\Theta_{2}}(T_{\mathbb{C}}M, \xi_{\mathbb{R}}, \eta_{\mathbb{C}}, a, b))\Big\}^{(4d)}.\nonumber
\end{align}
It is immediate that
\begin{align}
\overline{Q_{1}}(\tau)=&2\bigg\{e^{\frac{1}{24}E_{2}(\tau)[p_{1}(TM)-(a^{2}+2b^{2})p_{1}(\xi_{\mathbb{R}})]}\Big(\prod_{j=1}^{2d}x_{j}\frac{\theta'(0, \tau)}{\theta(x_{j}, \tau)}\Big)\\
&\times\Big(\frac{\theta_{1}(au, \tau)}{\theta_{1}(0, \tau)}\frac{\theta_{3}(bu, \tau)}{\theta_{3}(0, \tau)}\frac{\theta_{2}(bu, \tau)}{\theta_{2}(0, \tau)}\Big)\Big(\frac{\theta^{2}_{1}(0, \tau)}{\theta^{2}_{1}(\overline{u}, \tau)}\frac{\theta_{3}(\overline{u}, \tau)}{\theta_{3}(0, \tau)}\frac{\theta_{2}(\overline{u}, \tau)}{\theta_{2}(0, \tau)}\Big)\bigg\}^{(4d)},\nonumber\\
\overline{Q_{2}}(\tau)=&2\bigg\{e^{\frac{1}{24}E_{2}(\tau)[p_{1}(TM)-(a^{2}+2b^{2})p_{1}(\xi_{\mathbb{R}})]}\Big(\prod_{j=1}^{2d}x_{j}\frac{\theta'(0, \tau)}{\theta(x_{j}, \tau)}\Big)\\
&\times\Big(\frac{\theta_{1}(bu, \tau)}{\theta_{1}(0, \tau)}\frac{\theta_{3}(bu, \tau)}{\theta_{3}(0, \tau)}\frac{\theta_{2}(au, \tau)}{\theta_{2}(0, \tau)}\Big)
\Big(\frac{\theta^{2}_{2}(0, \tau)}{\theta^{2}_{2}(\overline{u}, \tau)}\frac{\theta_{3}(\overline{u}, \tau)}{\theta_{3}(0, \tau)}\frac{\theta_{1}(\overline{u}, \tau)}{\theta_{1}(0, \tau)}\Big)\bigg\}^{(4d)}.\nonumber
\end{align}
Therefore, $\overline{Q_{1}}(\tau)$ is a modular form of weight $2d$ over $\Gamma_{0}(2),$ while $\overline{Q_{2}}(\tau)$ is a modular form of weight $2d$ over $\Gamma^{0}(2).$
And we have
\begin{align}
\overline{Q_{1}}\Big(-\frac{1}{\tau}\Big)=\tau^{2d}\overline{Q_{2}}(\tau).
\end{align}

As in the proof of Theorem 3.1, we can get Theorem 4.1.
\end{proof}

\begin{cor}
When $\dim M=4,$ the following identity holds,
\begin{align}
&\Big\{2\cosh\Big(\frac{a}{2}c\Big)\frac{1}{\cosh^{2}\Big(\frac{1}{2}\overline{c}\Big)}\widehat{A}(TM, \nabla^{TM})\Big\}^{(4)}+2\Big\{2\cosh\Big(\frac{b}{2}c\Big)\cosh\Big(\frac{1}{2}\overline{c}\Big)\nonumber\\
&\times\widehat{A}(TM, \nabla^{TM})\Big\}^{(4)}=-\frac{1}{4}[p_{1}(TM)-(a^{2}+2b^{2})p_{1}(\xi_{\mathbb{R}})].\nonumber
\end{align}
\end{cor}

\begin{cor}
Assume that $M$ is a closed spin manifold and $\dim M=4,$ $p_{1}(TM)=(a^{2}+2b^{2})p_{1}(\xi_{\mathbb{R}})$ and if $b$ is an even number, $\eta\otimes\mathbb{C}=L\oplus\overline{L},$ then
\begin{align}
\int_{M}2\cosh\Big(\frac{a}{2}c\Big)\frac{1}{\cosh^{2}\Big(\frac{1}{2}\overline{c}\Big)}\widehat{A}(TM, \nabla^{TM})\nonumber
\end{align}
is an integer, where $\eta$ is a real orientable Euclidean vector bundle of rank $2$ over $M,$ $L$ is the complex line bundle associated with the spin$^{c}$ structure, $\overline{L}$ is its conjugate bundle.
\begin{proof}
Clearly,
\begin{align}
4\cosh\Big(\frac{b}{2}c\Big)\cosh\Big(\frac{1}{2}\overline{c}\Big)&=4\cdot\frac{e^{\frac{b}{2}c}-e^{-\frac{b}{2}c}}{2}\cdot\frac{e^{\frac{1}{2}\overline{c}}-e^{-\frac{1}{2}\overline{c}}}{2}\\
&=(e^{\frac{b}{2}c}-e^{-\frac{b}{2}c})(e^{\frac{1}{2}\overline{c}}-e^{-\frac{1}{2}\overline{c}})\nonumber\\
&=e^{\frac{b}{2}c}e^{\frac{1}{2}\overline{c}}(1-e^{-bc})(1-e^{-\overline{c}}).\nonumber
\end{align}
If $b$ is an even number, we have
\begin{align}
2\Big\{2\cosh\Big(\frac{b}{2}c\Big)\cosh\Big(\frac{1}{2}\overline{c}\Big)\widehat{A}(TM, \nabla^{TM})\Big\}[M]=\mathrm{Ind}\Big\{D^{c}\otimes\xi^{\otimes\frac{b}{2}}\otimes\Big(\mathbb{C}-(\xi^{*})^{\otimes b}\Big)\otimes\Big(\mathbb{C}-L^{*}\Big)\Big\}_{+},
\end{align}
is an integer.
So by Corollary 4.2, we get Corollary 4.3.
\end{proof}
\end{cor}

\begin{cor}
When $\dim M=8,$ the following identity holds,
\begin{align}
&\Big\{2\cosh\Big(\frac{a}{2}c\Big)\frac{1}{\cosh^{2}\Big(\frac{1}{2}\overline{c}\Big)}\widehat{A}(TM, \nabla^{TM})\Big\}^{(8)}-\Big\{2\cosh\Big(\frac{b}{2}c\Big)\cosh\Big(\frac{1}{2}\overline{c}\Big)\widehat{A}(TM, \nabla^{TM})\Big\}^{(8)}\nonumber\\
&-\frac{1}{16}\Big\{2\cosh\Big(\frac{b}{2}c\Big)\cosh\Big(\frac{1}{2}\overline{c}\Big)\widehat{A}(TM, \nabla^{TM})ch\big(\widetilde{(\xi^{\otimes b})_{\mathbb{R}}\otimes\mathbb{C}}-\widetilde{(\xi^{\otimes a})_{\mathbb{R}}\otimes\mathbb{C}}+3\widetilde{\eta_{\mathbb{C}}}\big)\Big\}^{(8)}\nonumber\\
&=-\frac{1}{24}\Big\{[p_{1}(TM)-(a^{2}+2b^{2})p_{1}(\xi_{\mathbb{R}})]2\cosh\Big(\frac{a}{2}c\Big)\frac{1}{\cosh^{2}\Big(\frac{1}{2}\overline{c}\Big)}\widehat{A}(TM, \nabla^{TM})\Big\}^{(8)}+\frac{1}{24}\nonumber\\
&\times \Big\{[p_{1}(TM)-(a^{2}+2b^{2})p_{1}(\xi_{\mathbb{R}})]2\cosh\Big(\frac{b}{2}c\Big)\cosh\Big(\frac{1}{2}\overline{c}\Big)\widehat{A}(TM, \nabla^{TM})\Big\}^{(8)}+\frac{1}{384}\nonumber\\
&\times  \Big\{[p_{1}(TM)-(a^{2}+2b^{2})p_{1}(\xi_{\mathbb{R}})]2\cosh\Big(\frac{b}{2}c\Big)\cosh\Big(\frac{1}{2}\overline{c}\Big)\widehat{A}(TM, \nabla^{TM})ch\big(\widetilde{(\xi^{\otimes b})_{\mathbb{R}}\otimes\mathbb{C}}\nonumber\\
&-\widetilde{(\xi^{\otimes a})_{\mathbb{R}}\otimes\mathbb{C}}+3\widetilde{\eta_{\mathbb{C}}}\big)\Big\}^{(8)}
+\frac{1}{18432}\Big\{([p_{1}(TM)-(a^{2}+2b^{2})p_{1}(\xi_{\mathbb{R}})])^{2}2ch\big(\widetilde{(\xi^{\otimes b})_{\mathbb{R}}\otimes\mathbb{C}}\nonumber\\
&-\widetilde{(\xi^{\otimes a})_{\mathbb{R}}\otimes\mathbb{C}}+3\widetilde{\eta_{\mathbb{C}}}\big)\Big\}^{(8)}.\nonumber
\end{align}
\end{cor}

\section{  Transgressed forms and modularities on odd-dimensional Riemannian manifolds }

In this section, $M$ be a $(4d-1)$-dimensional manifold.
Consider $\Phi_{1}(\nabla^{TM}, \nabla^{\xi}, \nabla^{\eta}, \tau)=\overline{Q_{1}}(\tau),$ $\Phi_{2}(\nabla^{TM}, \nabla^{\xi}, \nabla^{\eta}, \tau)=\overline{Q_{2}}(\tau).$
Applying the Chern-Weil theory, we can express $\Phi_{1}, \Phi_{2}$ as follows:
\begin{align}
\Phi_{1}(\nabla^{TM}, \nabla^{\xi}, \nabla^{\eta}, \tau)=&2e^{\frac{1}{24}E_{2}(\tau)[p_{1}(TM)-(a^{2}+2b^{2})p_{1}(\xi_{\mathbb{R}})]}\det\nolimits^{\frac{1}{2}}\Big(\frac{R^{TM}}{4\pi^{2}}\frac{\theta'(0, \tau)}{\theta(\frac{R^{TM}}{4\pi^{2}}, \tau)}\Big)\\
&\times\det\nolimits^{\frac{1}{2}}\Big(\frac{\theta_{1}(a\frac{R^{\xi}}{4\pi^{2}}, \tau)}{\theta_{1}(0, \tau)}\frac{\theta_{3}(b\frac{R^{\xi}}{4\pi^{2}}, \tau)}{\theta_{3}(0, \tau)}\frac{\theta_{2}(b\frac{R^{\xi}}{4\pi^{2}}, \tau)}{\theta_{2}(0, \tau)}\Big)\nonumber\\
&\times\det\nolimits^{\frac{1}{2}}\Big(\frac{\theta^{2}_{1}(0, \tau)}{\theta^{2}_{1}(\frac{R^{\eta}}{4\pi^{2}}, \tau)}\frac{\theta_{3}(\frac{R^{\eta}}{4\pi^{2}}, \tau)}{\theta_{3}(0, \tau)}\frac{\theta_{2}(\frac{R^{\eta}}{4\pi^{2}}, \tau)}{\theta_{2}(0, \tau)}\Big),\nonumber\\
\Phi_{2}(\nabla^{TM}, \nabla^{\xi}, \nabla^{\eta}, \tau)=&2e^{\frac{1}{24}E_{2}(\tau)[p_{1}(TM)-(a^{2}+2b^{2})p_{1}(\xi_{\mathbb{R}})]}\det\nolimits^{\frac{1}{2}}\Big(\frac{R^{TM}}{4\pi^{2}}\frac{\theta'(0, \tau)}{\theta(\frac{R^{TM}}{4\pi^{2}}, \tau)}\Big)\\
&\times\det\nolimits^{\frac{1}{2}}\Big(\frac{\theta_{1}(b\frac{R^{\xi}}{4\pi^{2}}, \tau)}{\theta_{1}(0, \tau)}\frac{\theta_{3}(b\frac{R^{\xi}}{4\pi^{2}}, \tau)}{\theta_{3}(0, \tau)}\frac{\theta_{2}(a\frac{R^{\xi}}{4\pi^{2}}, \tau)}{\theta_{2}(0, \tau)}\Big)\nonumber\\
&\times\det\nolimits^{\frac{1}{2}}\Big(\frac{\theta^{2}_{2}(0, \tau)}{\theta^{2}_{2}(\frac{R^{\eta}}{4\pi^{2}}, \tau)}\frac{\theta_{3}(\frac{R^{\eta}}{4\pi^{2}}, \tau)}{\theta_{3}(0, \tau)}\frac{\theta_{1}(\frac{R^{\eta}}{4\pi^{2}}, \tau)}{\theta_{1}(0, \tau)}\Big).\nonumber
\end{align}
Furthermore, we define
\begin{align}
\Phi_{3}(\nabla^{TM}, \nabla^{\xi}, \nabla^{\eta}, \tau)=&2e^{\frac{1}{24}E_{2}(\tau)[p_{1}(TM)-(a^{2}+2b^{2})p_{1}(\xi_{\mathbb{R}})]}\det\nolimits^{\frac{1}{2}}\Big(\frac{R^{TM}}{4\pi^{2}}\frac{\theta'(0, \tau)}{\theta(\frac{R^{TM}}{4\pi^{2}}, \tau)}\Big)\\
&\times\det\nolimits^{\frac{1}{2}}\Big(\frac{\theta_{1}(b\frac{R^{\xi}}{4\pi^{2}}, \tau)}{\theta_{1}(0, \tau)}\frac{\theta_{3}(a\frac{R^{\xi}}{4\pi^{2}}, \tau)}{\theta_{3}(0, \tau)}\frac{\theta_{2}(b\frac{R^{\xi}}{4\pi^{2}}, \tau)}{\theta_{2}(0, \tau)}\Big)\nonumber\\
&\times\det\nolimits^{\frac{1}{2}}\Big(\frac{\theta^{2}_{3}(0, \tau)}{\theta^{2}_{3}(\frac{R^{\eta}}{4\pi^{2}}, \tau)}\frac{\theta_{1}(\frac{R^{\eta}}{4\pi^{2}}, \tau)}{\theta_{1}(0, \tau)}\frac{\theta_{2}(\frac{R^{\eta}}{4\pi^{2}}, \tau)}{\theta_{2}(0, \tau)}\Big).\nonumber
\end{align}

Let $\nabla^{\eta}_{0}, \nabla^{\eta}_{1}$ be two Euclidean connections on $\eta,$ $A=\nabla^{\eta}_{1}-\nabla^{\eta}_{0}$ and $\nabla^{\eta}_{t}=t\nabla^{\eta}_{1}+(1-t)\nabla^{\eta}_{0}.$
By Lemma 2.3, we obtain
\begin{align}
&\Phi_{1}(\nabla^{TM}, \nabla^{\xi}, \nabla^{\eta}_{1}, \tau)-\Phi_{1}(\nabla^{TM}, \nabla^{\xi}, \nabla^{\eta}_{0}, \tau)\\
&=\frac{1}{8\pi^{2}}d\int_{0}^{1}\Phi_{1}(\nabla^{TM}, \nabla^{\xi}, \nabla^{\eta}_{t}, \tau)tr\Big[A\Big(\frac{\theta'_{2}(\frac{R_{t}^{\eta}}{4\pi^{2}},\tau)}{\theta_{2}(\frac{R_{t}^{\eta}}{4\pi^{2}}, \tau)}+\frac{\theta'_{3}(\frac{R_{t}^{\eta}}{4\pi^{2}},\tau)}{\theta_{3}(\frac{R_{t}^{\eta}}{4\pi^{2}}, \tau)}-2\frac{\theta'_{1}(\frac{R_{t}^{\eta}}{4\pi^{2}},\tau)}{\theta_{1}(\frac{R_{t}^{\eta}}{4\pi^{2}}, \tau)}\Big)\Big]dt.\nonumber
\end{align}
Then, we define
\begin{align}
&CS\Phi_{1}(\nabla^{TM}, \nabla^{\xi}, \nabla^{\eta}_{0}, \nabla^{\eta}_{1}, \tau)\\
&=\frac{1}{8\pi^{2}}\int_{0}^{1}\Phi_{1}(\nabla^{TM}, \nabla^{\xi}, \nabla^{\eta}_{t}, \tau)tr\Big[A\Big(\frac{\theta'_{2}(\frac{R_{t}^{\eta}}{4\pi^{2}},\tau)}{\theta_{2}(\frac{R_{t}^{\eta}}{4\pi^{2}}, \tau)}+\frac{\theta'_{3}(\frac{R_{t}^{\eta}}{4\pi^{2}},\tau)}{\theta_{3}(\frac{R_{t}^{\eta}}{4\pi^{2}}, \tau)}-2\frac{\theta'_{1}(\frac{R_{t}^{\eta}}{4\pi^{2}},\tau)}{\theta_{1}(\frac{R_{t}^{\eta}}{4\pi^{2}}, \tau)}\Big)\Big]dt,\nonumber
\end{align}
which is in $\Omega^{odd}(M, \mathbb{C})[[q^{\frac{1}{2}}]].$
Since $M$ is $4d-1$ dimensional, $\{CS\Phi_{1}(\nabla^{TM}, \nabla^{\xi}, \nabla^{\eta}_{0}, \nabla^{\eta}_{1}, \tau)\}^{(4d-1)}$ represents an element in $H^{4d-1}(M, \mathbb{C})[[q^{\frac{1}{2}}]].$
Similarly, we can compute the transgressed forms for $\Phi_{2}, \Phi_{3}$  and define
\begin{align}
&CS\Phi_{2}(\nabla^{TM}, \nabla^{\xi}, \nabla^{\eta}_{0}, \nabla^{\eta}_{1}, \tau)\\
&=\frac{1}{8\pi^{2}}\int_{0}^{1}\Phi_{2}(\nabla^{TM}, \nabla^{\xi}, \nabla^{\eta}_{t}, \tau)tr\Big[A\Big(\frac{\theta'_{1}(\frac{R_{t}^{\eta}}{4\pi^{2}},\tau)}{\theta_{1}(\frac{R_{t}^{\eta}}{4\pi^{2}}, \tau)}+\frac{\theta'_{3}(\frac{R_{t}^{\eta}}{4\pi^{2}},\tau)}{\theta_{3}(\frac{R_{t}^{\eta}}{4\pi^{2}}, \tau)}-2\frac{\theta'_{2}(\frac{R_{t}^{\eta}}{4\pi^{2}},\tau)}{\theta_{2}(\frac{R_{t}^{\eta}}{4\pi^{2}}, \tau)}\Big)\Big]dt,\nonumber\\
&CS\Phi_{3}(\nabla^{TM}, \nabla^{\xi}, \nabla^{\eta}_{0}, \nabla^{\eta}_{1}, \tau)\\
&=\frac{1}{8\pi^{2}}\int_{0}^{1}\Phi_{3}(\nabla^{TM}, \nabla^{\xi}, \nabla^{\eta}_{t}, \tau)tr\Big[A\Big(\frac{\theta'_{1}(\frac{R_{t}^{\eta}}{4\pi^{2}},\tau)}{\theta_{1}(\frac{R_{t}^{\eta}}{4\pi^{2}}, \tau)}+\frac{\theta'_{2}(\frac{R_{t}^{\eta}}{4\pi^{2}},\tau)}{\theta_{2}(\frac{R_{t}^{\eta}}{4\pi^{2}}, \tau)}-2\frac{\theta'_{3}(\frac{R_{t}^{\eta}}{4\pi^{2}},\tau)}{\theta_{3}(\frac{R_{t}^{\eta}}{4\pi^{2}}, \tau)}\Big)\Big]dt,\nonumber
\end{align}
which also lies in $\Omega^{odd}(M, \mathbb{C})[[q^{\frac{1}{2}}]]$ and its top component represents elements in $H^{4d-1}(M, \mathbb{C})[[q^{\frac{1}{2}}]].$
However we have the following results.

\begin{thm}
Let $M$ be a $4d-1$ dimensional manifold, $\nabla^{TM}$ be a connection on $TM,$ $\nabla^{\xi}$ be a connection on $\xi,$ $\eta$ be a two-dimensional oriented Euclidean real vector bundle with two Euclidean connections  $\nabla_{0}^{\eta}, \nabla_{1}^{\eta},$ then we have\\
1)
$\{CS\Phi_{1}(\nabla^{TM}, \nabla^{\xi}, \nabla^{\eta}_{0}, \nabla^{\eta}_{1}, \tau)\}^{(4d-1)}$
is a modular form of weight $2d$ over $\Gamma_0(2)$;\\
$\{CS\Phi_{2}(\nabla^{TM}, \nabla^{\xi}, \nabla^{\eta}_{0}, \nabla^{\eta}_{1}, \tau)\}^{(4d-1)}$
is a modular form of weight $2d$ over $\Gamma^0(2);$\\
$\{CS\Phi_{3}(\nabla^{TM}, \nabla^{\xi}, \nabla^{\eta}_{0}, \nabla^{\eta}_{1}, \tau)\}^{(4d-1)}$
is a modular form of weight $2d$ over $\Gamma_{\theta}.$\\
2) The following equalities hold,
\begin{align}
&\Big\{CS\Phi_{1}\Big(\nabla^{TM}, \nabla^{\xi}, \nabla^{\eta}_{0}, \nabla^{\eta}_{1}, -\frac{1}{\tau}\Big)\Big\}^{(4d-1)}
=\tau^{2d}\{CS\Phi_{2}(\nabla^{TM}, \nabla^{\xi}, \nabla^{\eta}_{0}, \nabla^{\eta}_{1}, \tau)\}^{(4d-1)},\nonumber\\
&\{CS\Phi_{2}(\nabla^{TM}, \nabla^{\xi}, \nabla^{\eta}_{0}, \nabla^{\eta}_{1}, \tau+1)\}^{(4d-1)}=\{CS\Phi_{3}(\nabla^{TM}, \nabla^{\xi}, \nabla^{\eta}_{0}, \nabla^{\eta}_{1}, \tau)\}^{(4d-1)}.\nonumber
\end{align}
\end{thm}

\begin{proof}
We denote $\frac{R_{t}^{\eta}}{4\pi^{2}}$ briefly by $z.$
It is easy to check that
\begin{align}
\frac{\theta'_{1}(z,-\frac{1}{\tau})} {\theta_{1}(z,-\frac{1}{\tau})}=2\pi\sqrt{-1}\tau z+\tau\frac{\theta'_{2}(\tau z,{\tau})} {\theta_{2}(\tau z,{\tau})},~
\frac{\theta'_{1}(z, \tau+1)} {\theta_{1}(z, \tau+1)}=\frac{\theta'_{1}(z, \tau)} {\theta_{1}(z, \tau)};\\
\frac{\theta'_{2}(z,-\frac{1}{\tau})} {\theta_{2}(z,-\frac{1}{\tau})}=2\pi\sqrt{-1}\tau z+\tau\frac{\theta'_{1}(\tau z,{\tau})} {\theta_{1}(\tau z,{\tau})},~
\frac{\theta'_{2}(z, \tau+1)} {\theta_{2}(z, \tau+1)}=\frac{\theta'_{3}(z, \tau)} {\theta_{3}(z, \tau)};\\
\frac{\theta'_{3}(z,-\frac{1}{\tau})} {\theta_{3}(z,-\frac{1}{\tau})}=2\pi\sqrt{-1}\tau z+\tau\frac{\theta'_{3}(\tau z,{\tau})} {\theta_{3}(\tau z,{\tau})},~
\frac{\theta'_{3}(z, \tau+1)} {\theta_{3}(z, \tau+1)}=\frac{\theta'_{2}(z, \tau)} {\theta_{2}(z, \tau)}.
\end{align}
Note that we only take $(4d-1)$-component, so we can get
\begin{align}
&\Big\{CS\Phi_{1}\Big(\nabla^{TM}, \nabla^{\xi}, \nabla^{\eta}_{0}, \nabla^{\eta}_{1}, -\frac{1}{\tau}\Big)\Big\}^{(4d-1)}
=\tau^{2d}\{CS\Phi_{2}(\nabla^{TM}, \nabla^{\xi}, \nabla^{\eta}_{0}, \nabla^{\eta}_{1}, \tau)\}^{(4d-1)};\\
&\{CS\Phi_{1}(\nabla^{TM}, \nabla^{\xi}, \nabla^{\eta}_{0}, \nabla^{\eta}_{1}, \tau+1)\}^{(4d-1)}=\{CS\Phi_{1}(\nabla^{TM}, \nabla^{\xi}, \nabla^{\eta}_{0}, \nabla^{\eta}_{1}, \tau)\}^{(4d-1)},\nonumber
\end{align}
\begin{align}
&\Big\{CS\Phi_{2}\Big(\nabla^{TM}, \nabla^{\xi}, \nabla^{\eta}_{0}, \nabla^{\eta}_{1}, -\frac{1}{\tau}\Big)\Big\}^{(4d-1)}
=\tau^{2d}\{CS\Phi_{1}(\nabla^{TM}, \nabla^{\xi}, \nabla^{\eta}_{0}, \nabla^{\eta}_{1}, \tau)\}^{(4d-1)};\\
&\{CS\Phi_{2}(\nabla^{TM}, \nabla^{\xi}, \nabla^{\eta}_{0}, \nabla^{\eta}_{1}, \tau+1)\}^{(4d-1)}=\{CS\Phi_{3}(\nabla^{TM}, \nabla^{\xi}, \nabla^{\eta}_{0}, \nabla^{\eta}_{1}, \tau)\}^{(4d-1)},\nonumber\\
&\Big\{CS\Phi_{3}\Big(\nabla^{TM}, \nabla^{\xi}, \nabla^{\eta}_{0}, \nabla^{\eta}_{1}, -\frac{1}{\tau}\Big)\Big\}^{(4d-1)}
=\tau^{2d}\{CS\Phi_{3}(\nabla^{TM}, \nabla^{\xi}, \nabla^{\eta}_{0}, \nabla^{\eta}_{1}, \tau)\}^{(4d-1)};\\
&\{CS\Phi_{3}(\nabla^{TM}, \nabla^{\xi}, \nabla^{\eta}_{0}, \nabla^{\eta}_{1}, \tau+1)\}^{(4d-1)}=\{CS\Phi_{2}(\nabla^{TM}, \nabla^{\xi}, \nabla^{\eta}_{0}, \nabla^{\eta}_{1}, \tau)\}^{(4d-1)}.\nonumber
\end{align}
\end{proof}
\section{ Acknowledgements}

The author was supported in part by  NSFC No.11771070. The author thanks the referee for his (or her) careful reading and helpful comments.

\vskip 1 true cm

%-----------------------------------------------------------------------------
%-----------------------------------------------------------------------------

\bigskip
\bigskip

\noindent {\footnotesize {\it S. Liu} \\
{School of Mathematics and Statistics, Changchun University of Science and Technology, Changchun 130022, China}\\
{Email: liusy719@nenu.edu.cn}

\noindent {\footnotesize {\it Y. Wang} \\
{School of Mathematics and Statistics, Northeast Normal University, Changchun 130024, China}\\
{Email: wangy581@nenu.edu.cn}

\clearpage
\section*{Statements and Declarations}

Funding: This research was funded by National Natural Science Foundation of China: No.11771070.\\

Competing Interests: The authors have no relevant financial or non-financial interests to disclose.\\

Author Contributions: All authors contributed to the study conception and design. Material preparation, data collection and analysis were performed by Siyao Liu and Yong Wang. The first draft of the manuscript was written by Siyao Liu and all authors commented on previous versions of the manuscript. All authors read and approved the final manuscript.\\

Availability of Data and Material: The datasets supporting the conclusions of this article are included within the article and its additional files.\\


\begin{thebibliography}{00}

\bibitem{AW} Alvarez-Gaum\'{e}, L., Witten, E., Gravitational anomalies, {\it Nuclear Phys. B}, {\bf 234}, 1983, 269-330.

\bibitem{L} Liu, K., Modular invariance and characteristic numbers, {\it Comm. Math. Phys.}, {\bf 174}, 1995, 29-42.

\bibitem{HZ1} Han, F., Zhang, W., Spin$^{c}$-manifold and elliptic genera, {\it C. R. Acad. Sci. Paris Ser. I.}, {\bf 336}, 2003, 1011-1014.

\bibitem{HZ2} Han, F., Zhang, W., Modular invariance, characteristic numbers and eta invariants, {\it J. Differential Geom.}, {\bf 67}, 2004, 257-288.

\bibitem{HH} Han, F., Huang, X., Even dimensional manifolds and generalized anomaly cancellation formulas, {\it Trans. Amer. Math. Soc.}, {\bf 359}, 2007, 5365-5382.

\bibitem{W1} Wang, Y., A note on generalized twisted anomaly cancellation formulas, {\it Acta Math. Sin.}, {\bf 26}, 2010, 1499-1508.

\bibitem{LW} Liu, K., Wang, Y., A note on modular forms and generalized anomaly cancellation formulas, {\it Sci. China Math.}, {\bf 56}, 2013, 55-65.

\bibitem{CH} Chen, Q., Han, F., Elliptic genera, transgression and loop space Chern-Simons forms, {\it Comm. Anal. Geom.}, {\bf 17}, 2009, 73-106.

\bibitem{W2} Wang, Y., Transgression and twisted anomaly cancellation formulas on odd dimensional manifolds, {\it J. Geom. Phys.}, {\bf 60}, 2010, 611-622.

\bibitem{A} Atiyah, M.F., K-Theory, Addison-Wesley, California, 1967.

\bibitem{H} Hirzebruch, F., Topological Methods In Algebraic Geometry, Springer-Verlag, Berlin, 1966.

\bibitem{Z} Zhang, W., Lectures On Chern-Weil Theory And Witten Deformations, World Scientific, Singapore, 2001.

\bibitem{C} Chandrasekharan, K., Elliptic Functions, Springer-Verlag, Berlin, 1985.



\end{thebibliography}
\end{document}